\documentclass[12pt,a4paper]{article}
\usepackage[german,ngerman,english]{babel}	
\usepackage[utf8]{inputenc}
\usepackage{amsmath,amssymb,amsfonts}
\usepackage{amsthm}
\usepackage{paralist}
\usepackage{mathtools}
\usepackage{color}
\usepackage{multicol}
\usepackage{authblk}
\usepackage{subcaption} 
\usepackage[font=scriptsize,labelfont=bf]{caption}
\usepackage{dsfont}
\usepackage{float}
\usepackage{natbib}

\newtheoremstyle{style}
  {0.5cm}                 
  {0.5cm}                 
  {\itshape}                         
  {}                         
  {\normalfont\bfseries}  
  {\normalfont {\bf : }} 
  {\newline}
  {}
  
\theoremstyle{style}

\newtheorem{theorem}{Theorem}[section]
\newtheorem{definition}{Definition}[section]
\newtheorem{lemma}{Lemma}[section]
\newtheorem{corollary}{Corollary}[section]
\newtheorem{prop}{Proposition}[section]
\newtheorem{assumption}{Assumption}[section]
\newtheorem{remark}{Remark}[section]

\newcommand{\R}{\mathbb{R}}

\begin{document}

\title{The Nonlocal Ramsey Model for an Interacting Economy}
\author[*]{L. Frerick}
\author[**]{G. M{\"u}ller-F{\"u}rstenberger}
\author[*]{E. W. Sachs}
\author[*]{L. Somorowsky}

\affil[*]{Department of Mathematics, University of Trier, Germany}
\affil[**]{Department of Economics, University of Trier, Germany}

\maketitle

\section{Introduction}

The Ramsey model, first introduced by F. Ramsey in 1928, has become a cornerstone in the economic growth theory. Combined with the Solow model in the 1950s by the economists \cite{cass} and \cite{koopmans}, it is still one of the most used neoclassical growth models. Due to its general and universal structure, it is widely used in various applications. Thus, many versions, including the effects of taxation, government spendings and population growth exist (see amongs others \cite{sorger}, \cite{barro90}, and \cite{acemoglu}).\\
The central idea in the Ramsey model is the endogenous saving rate, which is determined within the optimization process in the model via a lifetime utility maximization approach in the consuming sector. Though originally only time dependent, the Ramsey model has been spatialized in the last two decades in the context of the Economic Georgraphy, started by \cite{krugman1}. The first who introduced a spatial version of the Ramsey model was \cite{brito01}. This local version of the spatial Ramsey model assumes that the capital accumulation process in time and space can be modeled as a heat equation, with a local diffusion operator to describe the  mobility of capital across space. This assumption has not been contested until today. 
%
%
%
%
%

Economic geography depends not only on space, but also on the spreading behavior of production factors between several disjoint economies. In times of globalization and international trade agreements, it is also important in geographic economics to consider cross-border dependencies of production factors and economic welfare. By restricting the spatial domain of interest to a bounded domain $\Omega\subset\R^n$, we have naturally defined a border of an economy. Such a bounded economy could be seen as a country or a trade association. In contrast to the previous chapter, where we considered an unbounded spatial domain, we now have to introduce some boundary conditions in order to make the Ramsey problem well defined. As already mentioned in the introduction to the nonlocal spatial Ramsey model in Section \ref{NSRMWEPG}, these boundary constraints do not only act on the surface of the domain $\Omega$, but on a non-zero volume, the so called interaction, domain $\Omega_{\mathcal{I}}$. Whenever we understand the domain $\Omega$ as a bounded economy, $\Omega_{\mathcal{I}}$ can be interpreted as a trade-off set of production factors. By considering homogeneous Neumann-type volume constraints, we assume that the exchange of production factors between $\Omega$ and $\Omega_{\mathcal{I}}$ is balanced. We only allow capital and labor to leave the economy $\Omega$ and be replaced by the production factors in $\Omega_{\mathcal{I}}$. Moreover, we assume that neither production goods leave the economy nor consumption goods can be traded in the interaction set and that there is no production in $\Omega_{\mathcal{I}}$, hence $A_0=0$ on $\Omega_{\mathcal{I}}$. \\
We are not the first to restrict the spatial domain to a bounded set. For example, \cite{boucekkine13} define the spatial domain as the unit ball and \cite{aldashev}, following \cite{boucekkine13}, consider the parameterized circle as the interval $[0,2\pi]$. They do not define any boundary conditions but interpret the unit circle as the global economy. \cite{brock} and \cite{brock2013} consider a nonlocal model on an (arbitrary) compact interval. In their model, they do not consider any diffusion effects of the state variable but only time dependent spillover effects and so do not need to define any boundary conditions. A model close to our setting is described in \cite{anita}. Here, the authors consider a bounded space domain and introduce homogeneous Neumann boundary conditions. In the optimal control problem, they restrict the time line to a finite time horizon. Although their model is nonlocal as well, the quality of the nonlocality is different. Instead of considering nonlocal diffusion effects to model the mobility of capital, they include a pollution function in the capital accumulation equation. This pollution function is modeled as a partial differential equation with an integral term depending only on the capital function as right-hand side. However, the works listed above show that considering a bounded spatial domain and a finite time horizon is convenient for the economic application.\\

\section{A Nonlocal Vector Calculus}\label{NVC}

The fundamental theorem of calculus combines the concepts of the differential calculus and the integral calculus (cf. \citealp[p.304]{elstrodt}). Du, Gunzburger, Lehoucq, and Zhou introduce a new, nonlocal vector calculus. This theory aims at defining an analogon to the well known vector calculus for differential equations. In their papers, \citet{du1,du_ana,du} and \cite{gunzburger} derive a notion of nonlocal divergence and gradient operators and some fundamental relationships between the nonlocal operators and their derivatives. They are able to mimic the classical differential calculus to the framework of nonlocal operators and they prove identities like the Gaussian theorem or Green's identities for nonlocal diffusion equations. In that way, they make it possible to use the techniques of the analysis of common differential equations in the context of partial integro-differential equations.\\

In this section, we give a short introduction to the nonlocal vector calculus developed by \cite{du_ana,du1} and \citet{gunzburger}. We follow the notation in \cite{du_ana,du1} and \cite{delia2}, and denote $\Omega\subset\R^n$ as an open and bounded domain with sufficiently smooth boundary. Throughout this chapter, we assume $\Omega$ to be a Lipschitz domain. \\
The nonlocal vector calculus exploits the form of the nonlocal divergence and gradient operators. It is crucial to understand the {\em nonlocal diffusion operator}, commonly defined for a function $u:\R^n\to\R$ as
\begin{equation}\label{nonlocaldiff}
\mathcal{NL} u (x) := \int_{\R^n}(u(y)-u(x))\gamma(x,y) dy \qquad\mbox{ for } x\in\Omega,
\end{equation}
where $\Omega$ has nonzero volume and $\gamma$ denotes a nonnegative and symmetric kernel function, as a composition of nonlocal divergence and gradient operators, analogously to the local case.\\
Consider the two vector mappings $\nu,\ \alpha:\ \R^n\times\R^n\to\R^m$, with $\alpha$ antisymmetric. The {\em nonlocal divergence operator} $\mathcal{D}$ on $\nu$, $\mathcal{D}(\nu):\R^n\to\R$ is then defined as
\begin{equation}\label{divergence}
\mathcal{D}(\nu)(x):=\int_{\R^n} (\nu(x,y)+\nu(y,x))^T\alpha(x,y)dy \qquad \mbox{ for } x\in\R^n
\end{equation}
(c. \citealp[p.10]{du1}).\\
For a given mapping $u:\R^n\to\R$, \citet{du1} derive the the adjoint operator $\mathcal{D}^*$ corresponding to $\mathcal{D}$ with respect to the standard $L^2$ duality pairing as 
\begin{equation}\label{gradient}
\mathcal{D}^* (u)(x,y) = -(u(y)-u(x))\alpha(x,y)\qquad \mbox{ for } x,y\in\R^n.
\end{equation} 
The function $\mathcal{D}^*(u)$ maps from $\R^n\times\R^n$ to $\R^m$. The nonlocal adjoint operator $-\mathcal{D}^*$ can be interpreted as {\em nonlocal gradient operator}. Now, if the kernel function $\gamma$ in (\ref{nonlocaldiff}) is given as
\[\gamma = \alpha^T(\alpha),\]
the nonlocal diffusion operator in $(\ref{nonlocaldiff})$ can be represented as 
\[\mathcal{NL}(u) = -\frac{1}{2}\mathcal{D}(\mathcal{D}^*u).\]
Obviously $\mathcal{NL}$ is nonlocal since the evaluation of the operator in a point $x$ makes use of the evaluations of the function $u$ in points $\R^n \ni y\neq x$.

Due to this nonlocal character of $\mathcal{NL}$, it is not sufficient to consider boundary conditions that only act on the boundary $\partial \Omega$  of the set of interest, which is only a surface in $\R^n$. Instead, we have to introduce so called {\em volume constraints} which act on an {\em interaction domain} with nonzero volume. This interaction domain, denoted by $\Omega_{\mathcal{I}}\subset \R^n$, is the natural nonlocal extension of the surface-boundary of $\Omega$. Throughout we require that $\Omega_{\mathcal{I}} \cap \Omega=\emptyset$. \citet{du1} define the interaction domain as the set of all points in $\R^n\backslash \Omega$ that interact with points in $\Omega$, thus
\[\Omega_{\mathcal{I}} :=\{y\in\R^n\backslash \Omega:\ \alpha(x,y)\neq 0 \mbox{ for some } x\in\Omega\}.\]
Note that there is no assumption made about the geometric relation between the two sets $\Omega$ and $\Omega_{\mathcal{I}}$, however we restrict the definition of the interaction domain in our application to the case
\[\Omega_{\mathcal{I}} :=\{y\in\R^n\backslash \Omega:\ \|x-y\|_2\le \varepsilon,\  x\in\Omega\}.\]
The interaction of points in the domain of interest $\Omega$ with points in the interaction domain $\Omega_{\mathcal{I}}$ is modeled by a so called {\em nonlocal interaction operator} $\mathcal{V}$, an analogon to the local flux operator $\partial u/ \partial \overrightarrow{n}$. For a function $\nu:\R^n\times \R^n\to\R^m$, it is defined as
\[\mathcal{V}(\nu)(x) := -\int_{\Omega\cup\Omega_{\mathcal{I}}} (\nu(x,y)+\nu(y,x))^T\alpha(x,y)dy,\]
where $x$ is an element in $\Omega_{\mathcal{I}}.$ This operator can be interpreted as nonlocal flux from $\Omega$ into $\Omega_{\mathcal{I}}$.      \\
Note that in this setting, \citet{du1} provide a nonlocal version of the Gauss theorem, the integration by parts formula and the Green's identities. For further detail, we refer to \citet[Chapter 4]{du1}.

Not only does the nonlocal vector calculus theory provide some tools to analyze differential equations with nonlocal diffusion, the construction of the nonlocal diffusion operator as the composition of the nonlocal gradient and  divergence operator leads to the definition of a function space that is - under some circumstances - equivalent to the volume-constraint space of quadratic Lebesgue integrable functions 
\[L^2_c(\Omega\cup\Omega_{\mathcal{I}}) := \{u\in L^2(\Omega\cup\Omega_{\mathcal{I}}):\ E_c(u;0)=0\},\]
where the constraint functional $E_c$ is defined below in (\ref{Econstraint}). As appropriate function space for the (weak) solutions of the nonlocal (differential) equations, we consider the so called  {\em volume constrained nonlocal energy space,} which is given in the following definition:

\begin{definition}\label{vces}
The nonlocal energy space is defined by \citet{du_ana} as
\[V(\Omega\cup\Omega_{\mathcal{I}}) : = \{u\in L^2(\Omega\cup\Omega_{\mathcal{I}}):|||u|||<\infty\},\] 
where the {\em nonlocal energy norm} is given by
\[|||u|||:=\left(  \frac{1}{2}\int_{\Omega\cup\Omega_{\mathcal{I}}}\int_{\Omega\cup\Omega_{\mathcal{I}}}\mathcal{D}^*(u)(x,y)^T(\mathcal{D}^*(u)(x,y))\ dydx\right)^{\frac{1}{2}}.\]
Dealing with volume constraints, which are natural nonlocal expansions of the local boundary constraints, the {\em nonlocal volume-constrained energy space} is defined as
\begin{equation}\label{energyspace}
V_c(\Omega\cup\Omega_{\mathcal{I}}):=\{u\in V(\Omega\cup\Omega_{\mathcal{I}}):\ E_c(u;0)=0 \}, 
\end{equation}
where $E_c$ denotes the constraint functional, which depends on the volume constraint. \\
\end{definition}

\citet{du_ana} introduce two types of volume constraints, Dirichlet and Neumann type. 
In this paper, we restrict the consideration to the case of homogeneous Dirichlet constraints. This is due to application. We assume that the central planner follows a strategy and set the capital stock in the interaction domain to zero. In this way, he ...\\

Let $\Omega_{\mathcal{I}_d}$ denote the Dirichlet interaction set. The Dirichlet volume constraints are defined analogously to the local case as
\[u = g_d \mbox{ on } \Omega_{\mathcal{I}_d}.\] 
In this case, the constraint functional $E_c$, which defines the solution space, is given as
\begin{equation}\label{Econstraint}
E_c(u;g) = E_c(u;g_d) = g_d-\int_{\Omega_{\mathcal{I}_d}} u^2 \ dx
\end{equation}
for the given data $g_d$. This condition characterizes the solution space as defined in Definition \ref{vces} and ensures the existence of a unique solution \citep[pp.679]{du_ana}.\\

We have a closer look at the circumstances under which the nonlocal constrained energy space is equivalent to the $L^2_c$. According to \citet{delia2}, the kernel function $\gamma$ in (\ref{nonlocaldiff}) or (\ref{volume}) has to have the following properties:

\begin{prop}\label{prokernel}
For $x\in\Omega$, let $B_{\varepsilon}(x):=\{y\in \R^n:\ \|y-x\|_2\le \varepsilon\}$ be the $n$-dimensional ball with a given radius $\varepsilon>0$. Let the kernel function $\gamma$ satisfy the following properties:
\begin{enumerate}
\item $\gamma(x,y)\ge 0$ for all $ y\in\mathcal{B}_{\varepsilon}(x)$.
\item $\gamma(x,y)\ge\gamma_0> 0$ for all $ y\in\mathcal{B}_{\varepsilon/2}(x)$.
\item $\gamma(x,y) = 0$ for all $ y\in (\Omega\cup\Omega_{\mathcal{I}})\backslash \mathcal{B}_{\varepsilon}(x)$.
\item There exists a constant $\gamma_1>0$ such that
\[\gamma_1 \le \int_{(\Omega\cup\Omega_{\mathcal{I}})\cap\mathcal{B}_{\varepsilon}(x)}\gamma(x,y)\ dy \qquad \forall x\in\Omega.\]
\item There exists a constant $\gamma_2>0$ such that
\[ \qquad \int_{\Omega\cup\Omega_{\mathcal{I}}}\gamma^2(x,y)\ dy \le \gamma_2^2\qquad \forall x\in\Omega.\]
\end{enumerate}
Then, the nonlocal volume-constrained energy space $V_c(\Omega\cup\Omega_{\mathcal{I}})$ is equivalent to the volume constrained Lebesgue space,
\[L^2_c(\Omega\cup\Omega_{\mathcal{I}}):=\{u\in L^2(\Omega\cup\Omega_{\mathcal{I}}):\ E_c(u;0)=0\}.\]
Hence, there exist some constants $C_1$ and $C_2$, both positive, such that
\begin{equation}\label{norm_equivalnece}
C_1\|u\|_{L^2(\Omega\cup\Omega_{\mathcal{I}})} \le |||u||| \le C_2\|u\|_{L^2(\Omega\cup\Omega_{\mathcal{I}})}\qquad \forall u\in V_c(\Omega\cup\Omega_{\mathcal{I}}).
\end{equation}
Moreover, $V_c(\Omega\cup\Omega_{\mathcal{I}})$ endowed with the norm $|||\cdot|||$ is a Hilbert space. 
\end{prop}
The proof is given by \citet[p.684]{du_ana}.\\

The dual space of $V_c(\Omega\cup\Omega_{\mathcal{I}})$ with respect to the standard $L^2(\Omega\cup\Omega_{\mathcal{I}})$ pairing is denoted by $V'_c(\Omega\cup\Omega_{\mathcal{I}})$. If $\gamma$ satisfies all properties of Proposition \ref{prokernel}, this dual is equivalent to $L^2_c(\Omega\cup\Omega_{\mathcal{I}})$ as well. The norm on $V'_c(\Omega\cup\Omega_{\mathcal{I}})$ can naturally be defined as
\[\|f\|_{V'_c(\Omega\cup\Omega_{\mathcal{I}})} := \sup_{u\in V_c(\Omega\cup\Omega_{\mathcal{I}}),\ u\neq 0} \frac{\int_{\Omega\cup\Omega_{\mathcal{I}}} fu\ dx}{|||u|||}.\]
Especially for the kernel function considered in this context, it is true that $V'_c(\Omega\cup\Omega_{\mathcal{I}})$ is equivalent to $L^2_c(\Omega\cup\Omega_{\mathcal{I}})$ such that $\|f\|_{V'_c(\Omega\cup\Omega_{\mathcal{I}})}\le \|f\|_{L^2(\Omega\cup\Omega_{\mathcal{I}})}$
(c. \citealp[p.248]{delia}).\\

The spatial Ramsey model is defined over a space-time cylinder, hence we have to consider the time dependent spaces
\[L^2(0,T;V_c(\Omega\cup\Omega_{\mathcal{I}})):=\{u(\cdot,t)\in V_c(\Omega\cup\Omega_{\mathcal{I}}):\ |||u(\cdot,\cdot)|||\in L^2(0,T)\},\]
and
\[L^2(0,T;V'_c(\Omega\cup\Omega_{\mathcal{I}})):=\{u(\cdot,t)\in V'_c(\Omega\cup\Omega_{\mathcal{I}}):\ \|u(\cdot,\cdot)\|_{V'_c}\in L^2(0,T)\},\]
for $T>0$ respectively \citep[p.10]{delia2}.\\
For functions that are weakly differentiable according to time, we define the space
\[H^1(0,T;V_c(\Omega\cup\Omega_{\mathcal{I}})):=\{u\in L^2(0,T;V_c(\Omega\cup\Omega_{\mathcal{I}})):\ \frac{\partial u}{\partial t }\in L^2(0,T;V_c'(\Omega\cup\Omega_{\mathcal{I}}))\}.\]

We can conclude later that the weak solution of our problem is not only weakly differentiable, but also continuous in the time variable. Hence, the function space where we expect our weak solution to live in, is
\[ \mathcal{C}(0,T;V_c(\Omega\cup\Omega_{\mathcal{I}})) \cap H^1(0,T;V_c(\Omega\cup\Omega_{\mathcal{I}})).\]
This intersection has to be understood as subspace of $ \mathcal{C}(0,T;V_c(\Omega\cup\Omega_{\mathcal{I}})).$


\section{The Weak Solution over Bounded Spatial Domains}\label{EWSBD}

We embed the nonlocal spatial Ramsey model on bounded spatial domains to the nonlocal vector calculus. We base our assumptions on the papers of \citet{delia} and \citet{delia2}. First, we discuss the nonlocal vector calculus with respect to applicability to our model. Afterwards, we derive an existence and several regularity results of the weak solution of the nonlocal spatial Ramsey model on bounded domains.

\subsection{Embedding the Nonlocal Spatial Ramsey Model in the Nonlocal Vector Calculus}

We assume that the domain $\Omega\subset\R^n$ has a (at least piecewise) smooth boundary and satisfies the cone condition. The interaction domain $\Omega_{\mathcal{I}}$ and the nonlocal closure $\Omega\cup\Omega_{\mathcal{I}}$ are assumed to have the same properties. We set $\alpha:\R^n\times \R^n\to \R$,
\begin{equation*}
\begin{split}
&\alpha(x,y)=\alpha_\varepsilon(x,y):=  \\
&\operatorname{sign}(\|x\|_2-\|y\|_2) \left(\frac{1}{\sqrt{(2\pi\sigma^2)^n}} \exp\left( -\frac{1}{2} (x-y)^T\Sigma_\sigma^{-1}(x-y) \right)\mathds{1}_{B_\varepsilon(x)}(y) \right)^{\frac{1}{2}},
\end{split}
\end{equation*}
for $\varepsilon,\ \sigma>0$, and a covariance matrix $\Sigma_\sigma\in\R^{n\times n}$. The parameter $\varepsilon$ will be referred to as the interaction radius. The nonlocal interaction domain is then given as
\[\Omega_{\mathcal{I}} :=\{y\in\R^n\backslash \Omega:\ \|y-x\|_2<\varepsilon\mbox{ for } x\in\Omega\}.\]
Note, that we do explicitly allow to choose $\sigma\neq \varepsilon$. But since  $\varepsilon$ is the only parameter that is important for estimates and the calculations below, we only keep the dependence of $\alpha_\varepsilon$ on $\sigma$ in mind and do not use it in the notation.  \\
The function $\alpha_\varepsilon$ is obviously antisymmetric and we can easily calculate that the kernel function in the nonlocal Ramsey model indeed has the form
 \[\Gamma_\varepsilon(x,y) = \alpha_\varepsilon^2(x,y) = \frac{1}{\sqrt{(2\pi\sigma^2)^n}} \exp\left( -\frac{1}{2} (x-y)^T\Sigma_\sigma^{-1}(x-y) \right)\mathds{1}_{B_\varepsilon(x)}(y).\]
We assume that the covariance matrix is a diagonal matrix with equal entries, 
 \[\Sigma_\sigma:=\begin{bmatrix}
 \sigma^2 &&\\
 &\ddots&\\
 && \sigma^2
 \end{bmatrix},\]
such that $\det(\Sigma_\sigma)=\sigma^{2n}$. Then, the kernel function satisfies all properties required in order to fit the nonlocal vector calculus, which we prove with the following Lemma:

\begin{lemma}\label{kernelProps}
The kernel function $\Gamma_\varepsilon$ satisfies all properties of Proposition \ref{prokernel}. 
\end{lemma}
\begin{proof}
The kernel function $\Gamma_\varepsilon$ is given by
\[\Gamma_\varepsilon(x,y) := \frac{1}{\sqrt{(2\pi\sigma^2)^n}}\exp\left( - \frac{\|x-y\|_2^2}{2\sigma^2}\right)\mathds{1}_{B_\varepsilon(x)}(y),\]
which is obviously symmetric. We go on checking all properties as in Proposition \ref{prokernel}.\\
\begin{itemize}
\item[(1)],(2) Let $0<\eta\le \varepsilon$. For all $y\in B_{\eta}(x)$ it is true that
\begin{equation*}
\begin{split}
\Gamma_\varepsilon(x,y) 
&\ge  \frac{1}{\sqrt{(2\pi\sigma^2)^n}}\exp\left( - \frac{\eta^2}{\sigma^2}\right)>0
\end{split}
\end{equation*}

\item[(3)] The third property follows with the definition of the indicator function.
\item[(4)] For $x\in\Omega$, we calculate

\begin{equation*}
\begin{split}
\int_{(\Omega \cup \Omega_{\mathcal{I}}) \cap B_{\varepsilon}(x)} \Gamma_\varepsilon(x,y)\ dy & = \int_{B_\varepsilon(x)}  \frac{1}{\sqrt{(2\pi\sigma^2)^n}}\exp\left( - \frac{\|x-y\|_2^2}{2\sigma^2}\right) \ dy\\
& \ge \int_{B_\varepsilon(x)}  \frac{1}{\sqrt{(2\pi\sigma^2)^n}}\exp\left( \frac{- \varepsilon^2}{\sigma^2}\right) \ dy\\
& = c_n\varepsilon^n \frac{1}{\sqrt{(2\pi\sigma^2)^n}} \exp\left( - \frac{\varepsilon^2}{\sigma^2}\right)>0
\end{split}
\end{equation*}
where $c_n$ denotes the volume of the unit sphere in $\R^n$.
\item[(5)] For the last property, we calculate for $x\in\Omega$,
\begin{equation*}
\begin{split}
\int_{\Omega\cup \Omega_{\mathcal{I}}}\Gamma_\varepsilon^2(x,y) \ dy &= \int_{\Omega \cup \Omega_{\mathcal{I}}} \frac{1}{(2\pi\sigma^2)^n}\exp\left( - \frac{\|x-y\|_2^2}{\sigma^2} \right) \mathds{1}_{B_\varepsilon(x)}(y) \ dy\\
& = \int_{B_\varepsilon(x)}  \frac{1}{(2\pi\sigma^2)^n}\exp\left( - \frac{\|x-y\|_2^2}{\sigma^2} \right)\ dy\\
& \le  \int_{B_\varepsilon(x)}  \frac{1}{(2\pi\sigma^2)^n}\exp\left( - \frac{0}{\sigma^2} \right)\ dy\\
&  = \frac{c_n\varepsilon^n}{(2\pi\sigma^2)^n}<\infty,
\end{split}
\end{equation*}
which completes the proof.
\end{itemize}

\end{proof}

We now have all at hand to define the nonlocal spatial Ramsey model with endogenous productivity growth under a finite time horizon and an open, but bounded spatial domain:\\

For a given initial condition $k_0\in V_c(\Omega)$, find an optimal control $c^*\in \mathcal{U}_{ad}\subset L^2(0,T;V'_c(\Omega))$ and an optimal state $k^*\in  \mathcal{C}(0,T;V_c(\Omega\cup\Omega_{\mathcal{I}})) \cap H^1(0,T;V_c(\Omega\cup\Omega_{\mathcal{I}}))$, such that 

\begin{equation}\label{Jnonlocal}
\begin{split}
\mathcal{J}(k,c)&  := \int_0^T\int_{\Omega} -U(c(x,t))e^{-\tau t-\gamma \|x\|^2_2} \ dxdt + \frac{1}{2\rho}\int_\Omega (k(x,T)-k_T(x))^2\ dx
\end{split}
\end{equation}

is minimized subject to $k$ and $c$ satisfying 

\begin{equation}\label{cons_nonlocal}
\begin{split}
\frac{\partial k}{\partial t} - \beta\mathcal{NL}_\varepsilon (k) +\delta k & = \mathcal{P}(k) - c \hspace*{1cm}\mbox{  on  }\Omega\times (0,T),\\
k &= 0 \hspace*{2.35cm} \mbox{  on  } \Omega_{\mathcal{I}}\times(0,T),\\[2mm]
c \in \mathcal{U}_{ad},\  k&\ge 0 \hspace*{2.35cm} \mbox{ on } \Omega\times(0,T),\\[2mm]
k(\cdot,0) &= k_0>0 \hspace*{1.5cm}\mbox{  in  } \Omega,
\end{split}
\end{equation}

where $\mathcal{NL}_\varepsilon$ is given as in (\ref{nonlocaldiff}) with the  kernel function $\gamma:=\Gamma_\varepsilon$ and the nonlocal productivity-production operator $\mathcal{P}$ is defined as
\[\mathcal{P}(k)(x,t):=A_0(x)\exp\left(\frac{\int_{\Omega\cup\Omega_{\mathcal{I}}} \phi(k(y,t))\Gamma_\mu(x,y)\ dy}{\int_{\Omega\cup\Omega_{\mathcal{I}}} \phi(k(y,t))\Gamma_\varepsilon(x,y)\ dy + \xi}  t\right)p(k(x,t))\]
for a non-negative, continuous, and real valued function $\phi$ and constants $\xi>0$ and $0<\mu<\varepsilon$. The set $\mathcal{U}_{ad}$ denotes the set of feasible controls, which will be described in more detail later.

\begin{remark}
Note that the definition of the kernel function as a truncated Gaussian density function, depending on an indicator function which is determined by the parameters $\mu$ and $\varepsilon$, guarantees that
\[\frac{\int_{\Omega\cup\Omega_{\mathcal{I}}} \phi(k(y,t))\Gamma_\mu(x,y)\ dy}{\int_{\Omega\cup\Omega_{\mathcal{I}}} \phi(k(y,t))\Gamma_\varepsilon(x,y)\ dy + \xi} \le 1.\]
\end{remark}

\subsection{Existence of a Weak Solution}

The main part of this section is the proof of the existence of a weak solution of the problem (\ref{cons_nonlocal}). We apply an abstract existence result by \citet{wloka} and the fixed point theorem of Banach. To do so, we exploit the Lipschitz continuity of the productivity-production operator $\mathcal{P}$ and refer to a result by \citet[p.686]{du_ana} which states that the linear nonlocal diffusion problem with homogeneous Dirichlet-type volume constraints has a unique weak solution in the volume constrained energy space.\\

Throughout this section, we assume that the depreciation rate $\delta$, the initial productivity distribution function $A_0$, the nominal function $\phi$, and the nonlinear production function $p$ satisfy the following assumptions:

\begin{assumption}
Consider the functions $p:\R\to\R_+$, $\phi:\R\to\R_+$, and $A_0:\R^n\to\R_+$. Then, we assume
\begin{itemize}
\item The neoclassical production function $p$ is bounded from above by a constant $M_p>0$, it satisfies $p(0)=0$ and is Lipschitz continuous with Lipschitz constant $L_p>0$.
\item The nominal function $\phi$ is Lipschitz continuous with constant $L_\phi>0$.
\item The initial productivity distribution function $A_0$ is in $L^\infty(\Omega)$.
\end{itemize}
\end{assumption}

We derive the weak formulation of the system (\ref{cons_nonlocal}), i.e. we multiply the state equation with a test function $\varphi \in C(0,T; V_c(\Omega\cup\Omega_{\mathcal{I}}))$ and integrate over $\Omega\times(0,T)$ which yields

\begin{equation}
\begin{split}
\int_0^T \int_{\Omega} k_t \varphi \ dx dt & -  \int_0^T\int_{\Omega}\mathcal{NL}_{\varepsilon}(k)\ \varphi \ dx dt +\delta \int_0^T \int_{\Omega}  k \varphi \ dxdt\\
& = \int_0^T \int_{\Omega} (\mathcal{P}(k) - c)\varphi \ dxdt.
\end{split}
\end{equation}
Applying the nonlocal Green's first identity (see \citet[p.676]{du_ana}) then gives us

\begin{equation}\label{PIDE_weak}
\begin{split}
\int_0^T \int_{\Omega} k_t \varphi \ dx dt & + \frac{1}{2}\int_0^T \int_{\Omega\cup\Omega_{\mathcal{I}}} \int_{\Omega\cup\Omega_{\mathcal{I}}}\mathcal{D}^*(k)^T\mathcal{D}^*(\varphi)\ \ dydx dt +\delta \int_0^T \int_{\Omega}  k \varphi \ dxdt\\
& = \int_0^T \int_{\Omega} (\mathcal{P}(k) - c)\varphi \ dxdt.
\end{split}
\end{equation}

Analogously to the spatially unbounded case, this weak formulation of the capital accumulation equation gives rise to the following definition of a bilinear form 
\[{\bf a}:V_c(\Omega\cup\Omega_{\mathcal{I}})\times V_c(\Omega\cup\Omega_{\mathcal{I}}) \to\R ,\]

\begin{equation}\label{a_bounded}
{\bf a}(u,v) := \frac{1}{2}\int_{\Omega\cup\Omega_{\mathcal{I}}} \int_{\Omega\cup\Omega_{\mathcal{I}}}\mathcal{D}^*(u)^T\mathcal{D}^*(v) \ dydx + \delta \int_{\Omega}  u \ v \ dx.
\end{equation} 

We prove the coercivity and continuity of ${\bf a}$ in the following Lemma.

\begin{lemma}\label{a_properties}
The bilinear form ${\bf a}$ is coercive and continuous, hence there exists a constant $c_1>0$ such that
\begin{itemize}
\item[(i)] $|{\bf a} (u,v)| \le c_1|||u|||\ |||v|||$,
\item[(ii)] ${\bf a}(u,u) \ge |||u||| $.
\end{itemize}
\end{lemma}

\begin{proof}
\begin{itemize}
\item[(i)] Choose $u,v\ \in V_c(\Omega\cup \Omega_{\mathcal{I}})$. Then,
\begin{equation*}
\begin{split}
|{\bf a}(u,v)| & = \left| \frac{1}{2}\int_{\Omega\cup\Omega_{\mathcal{I}}} \int_{\Omega\cup\Omega_{\mathcal{I}}}\mathcal{D}^*(u)(x,y)^T\mathcal{D}^*(v)(x,y) \ dydx   + \delta \int_{\Omega}  u\ v \ dx \right|\\
&\le \frac{1}{2} \int_{\Omega\cup\Omega_{\mathcal{I}}} \int_{\Omega\cup\Omega_{\mathcal{I}}} |\mathcal{D}^*(u)(x,y)^T\mathcal{D}^*(v)(x,y)| \ dydx + \delta \int_{\Omega} |uv|\ dx.
\end{split}
\end{equation*}
We use the Cauchy Schwartz inequality which yields together with the norm equivalence (\ref{norm_equivalnece}),
\begin{equation*}
\begin{split}
|{\bf a}(u,v)| &\le \frac{1}{2}|||u|||\ |||v||| + \delta \int_\Omega|u\ v| \ dx\\ &\le \frac{1}{2}|||u|||\ |||v||| + \delta  \|u\|_{L^2(\Omega)}\|v\|_{L^2(\Omega)}\\
&\le \left( \frac{1}{2} + \frac{\delta}{C_1} \right)|||u|||\ |||v|||
\end{split}
\end{equation*}
for the constant $C_1>0$ from (\ref{norm_equivalnece}).

\item[(ii)] Applying the Poincare inequality, we have
\begin{equation*}
\begin{split}
{\bf a}(u,u) & = \frac{1}{2} \int_{\Omega\cup\Omega_{\mathcal{I}}}\int_{\Omega\cup\Omega_{\mathcal{I}}} \left( \mathcal{D}^*(u)(x,y)\right)^2 \ dydx + \delta \int_{\Omega}u^2(x) \ dx\\
&\ge |||u|||^2 + \delta \|u\|^2_{L^2(\Omega)}\\
&\ge  |||u|||^2.
\end{split}
\end{equation*}
\end{itemize}
\end{proof}

We can now prove the existence of a weak solution of the capital accumulation equation. We apply Banach's fixed point theorem and an existence result for a linear but inhomogeneous nonlocal diffusion equation given by \citet[Theorem 5.1, p.686]{du_ana}.

\begin{theorem}\label{Existence_semilinear_bounded}
For a given $c \in L^2(0,T;V_c'(\Omega)) $ and $k_0\in V_c(\Omega),$ the problem (\ref{PIDE_weak}) with $k(0,x) = k_0(x)$ in $\Omega$ and $k=0$ on $ \Omega_{\mathcal{I}}$,  has a unique weak solution $k^*\in C(0,T;V_c(\Omega\cup\Omega_{\mathcal{I}}))\cap H^1(0,T;V_c(\Omega\cup\Omega_{\mathcal{I}}))$.
\end{theorem}

\begin{proof}
Let $\mathcal{S}: C(0,T;V_c(\Omega\cup\Omega_{\mathcal{I}}))\cap H^1(0,T;V_c(\Omega\cup\Omega_{\mathcal{I}})) \to C(0,T;V_c(\Omega\cup\Omega_{\mathcal{I}}))\cap H^1(0,T;V_c(\Omega\cup\Omega_{\mathcal{I}}))$ be the operator that maps a function $v$ to the unique function $k$ that satisfies $k(x,0)=k_0(x)$ on $\Omega$, $k=0$ on $\Omega_{\mathcal{I}}\times(0,T)$ and that solves the weak formulation of the linear equation
\begin{equation*}
\begin{split}
\int_0^T \int_{\Omega} k_t(x,t)\varphi(x,t)\ dxdt + &\int_0^T {\bf a}(k(\cdot,t),\varphi(\cdot,t)) dt =\\
 &\int_0^T \int_{\Omega} (\mathcal{P}(v)(x,t) - c(x,t))\varphi(x,t)\ dxdt
\end{split}
\end{equation*}
for all $\varphi \in \mathcal{C}(0,T;V_c(\Omega\cup\Omega_{\mathcal{I}}))$. We fix $T^*\in(0,T)$ sufficiently small and consider the difference $S(v_1)-S(v_2)$ for two arbitrary functions $v_1,\ v_2 \in C(0,T^*;V_c(\Omega\cup\Omega_{\mathcal{I}}))$ with $S(v_1)=k_1$ and $S(v_2)=k_2$. We choose the test function $k_1-k_2 \in  \mathcal{C}(0,T^*;V_c(\Omega\cup\Omega_{\mathcal{I}}))\cap H^1(0,T^*;V_c(\Omega\cup\Omega_{\mathcal{I}}))$. Then, $k_1-k_2$ solves
\begin{equation*}
\begin{split}
&\int_0^t \int_{\Omega} (k_1-k_2)_t(x,s)(k_1-k_2)(x,s) \ dx + {\bf a}(k_1-k_1,k_1-k_2)(s)ds =\\
& \hspace*{4cm} \int_0^t \int_{\Omega} \left( \mathcal{P}(v_1)(s,x) - \mathcal{P}(v_2)(x,s)\right)(k_1-k_2)(x,s) \ dxds,
\end{split}
\end{equation*}
for all $t\in[0,T^*]$ according to \citet[p.686]{du_ana}.\\
We estimate the left-hand side using the coercivity property of the bilinear form ${\bf a}$ according to Lemma \ref{a_properties}:
\begin{equation*}
\begin{split}
&\int_0^t \int_{\Omega} (k_1-k_2)_t(x,s)(k_1-k_2)(x,s) \ dx + {\bf a}(k_1-k_1,k_1-k_2)(s)ds \\
&\ge \int_0^t \int_{\Omega} (k_1-k_2)_t(x,s) (k_1-k_2)(x,s) \ dx + |||k_1-k_2(s)|||^2 ds\\
& = \frac{1}{2} \|k_1-k_2(t)\|^2_{L^2(\Omega)} +  \int_0^t |||k_1-k_2(s)|||^2 ds.
\end{split}
\end{equation*}

For the right-hand side, we exploit the Lipschitz property of $\mathcal{P}$ on bounded spatial domains as follows: First, we apply the H\"older inequality,

\begin{equation*}
\begin{split}
&\int_0^t \int_{\Omega} \left( \mathcal{P}(v_1)(x,s)-\mathcal{P}(v_2)(x,s)  \right)\left(k_1(x,s)-k_2(x,s)\right)\ dxds \\
&\le \int_0^t \|\mathcal{P}(v_1)(\cdot,s)-\mathcal{P}(v_2)(\cdot,s)\|_{L^2(\Omega)}\|k_1(\cdot,s)-k_2(\cdot,s)\|_{L^2(\Omega)} ds:= (\#).\quad \\
\end{split}
\end{equation*}

Now, we add a `clever zero' and calculate

\begin{equation*}
\begin{split}
(\#)&=  \int_0^t \|\mathcal{P}(v_1)(\cdot,s)-P(v_1)(\cdot,s)p(v_2(\cdot,s)) + P(v_1)(\cdot,s)p(v_2(\cdot,s))-\mathcal{P}(v_2)(\cdot,s)\|_{L^2(\Omega)}\\
&\hspace*{1cm} \cdot\|k_1(\cdot,s)-k_2(\cdot,s)\|_{L^2(\Omega)} ds \\
&\le  \int_0^t \|\mathcal{P}(v_1)(\cdot,s)-P(v_1)(\cdot,s)p(v_2(\cdot,s))\|_{L^2(\Omega)}\|k_1(\cdot,s)-k_2(\cdot,s)\|_{L^2(\Omega)} ds\\
& + \int_0^t\| P(v_1)(\cdot,s)p(v_2(\cdot,s))-\mathcal{P}(v_2)(\cdot,s)\|_{L^2(\Omega)}\|k_1(\cdot,s)-k_2(\cdot,s)\|_{L^2(\Omega)} ds\\
&\le \int_0^t \| P(v_1)(\cdot,s)\|_{L^\infty(\Omega)}\|p(v_1(\cdot,s)-p(v_2(\cdot,s))\|_{L^2(\Omega)}\|k_1-k_2(s)\|_{L^2(\Omega)} ds\\
&+ \int_0^t \|p(v_2(\cdot,s))\|_{L^\infty(\Omega)}\|P(v_1)(\cdot,s)-P(v_2)(\cdot,s)\|_{L^2(\Omega)}\|k_1(\cdot,s)-k_2(\cdot,s)\|_{L^2(\Omega)} ds\\
&\le \int_0^t L_p \| P(v_1)(\cdot,s)\|_{L^\infty(\Omega)}\|v_1(\cdot,s)-v_2(\cdot,s)\|_{L^2(\Omega)}\|k_1(\cdot,s)-k_2(\cdot,s)\|_{L^2(\Omega)} ds\\
& + \int_0^t M\ \|P(v_1)(\cdot,s)-P(v_2)(\cdot,s)\|_{L^2(\Omega)} \|k_1-k_2(s)\|_{L^2(\Omega)} ds.
\end{split}
\end{equation*}

We have a closer look at the terms $\| P(v_1)(\cdot,s)\|_{L^\infty(\Omega)}$ and $\|P(v_1)(\cdot,s)-P(v_2)(\cdot,s)\|_{L^2(\Omega)}$. Again, we exploit the boundedness of the fraction in the exponential term of $\mathcal{P}$, which is bounded by
\[\frac{\int_{\Omega\cup\Omega_{\mathcal{I}}} \phi(v(y,s))\Gamma_{\mu}(x,y)\  dy}{\int_{\Omega\cup\Omega_{\mathcal{I}}} \phi(v(y,s))\Gamma_{\varepsilon}(x,y)\  dy +\xi}\le 1\]
by the definition of the indicator function, and the monotonicity of the integral. Hence, we can estimate

\begin{equation*}
\begin{split}
\|P(v)(\cdot,s)\|_{L^{\infty}(\Omega)}&= ess\ \sup_{x\in\Omega}|P(v)(x,s)| \\
&= ess\ \sup _{x\in\Omega} \left| A_0(x) \exp\left( \frac{\int_{\Omega\cup\Omega_{\mathcal{I}}} \phi(v(y,s))\Gamma_{\mu}(x,y)\  dy}{\int_{\Omega\cup\Omega_{\mathcal{I}}} \phi(v(y,s))\Gamma_{\varepsilon}(x,y)\  dy +\xi}\ s\right)\right|\\
&\le \|A_0\|_{L^{\infty}(\Omega)} \exp(s),
\end{split}
\end{equation*}
and

\begin{equation*}
\begin{split}
&\|P(v_1)(\cdot,s)-P(v_2)(\cdot,s)\|_{L^2(\Omega)} \le   \|A_0\|_{L^{\infty}(\Omega)}\ \left\| \exp\left(\frac{\int_{\Omega\cup\Omega_{\mathcal{I}}}\phi(v_1(y,s))\Gamma_{\mu}(\cdot,y)\  dy}{\int_{\Omega\cup\Omega_{\mathcal{I}}}\phi(v_1(y,s))\Gamma_{\varepsilon}(\cdot, y) \  dy+\xi}\ s\right) \right.\\
&\left. \hspace*{6cm}- \exp\left(\frac{\int_{\Omega\cup\Omega_{\mathcal{I}}}\phi(v_2(y,s))\Gamma_{\mu}(\cdot,y)\  dy}{\int_{\Omega\cup\Omega_{\mathcal{I}}}\phi(v_2(y,s))\Gamma_{\varepsilon}(\cdot,y) \  dy+\xi}\ s\right) \right\|_{L^2( \Omega)}
\end{split}
\end{equation*}

In order to keep a compact representation, we define the operator
\[\Phi_\nu(v)(x,s):= \int_{\Omega\cup\Omega_{\mathcal{I}}} \phi(v(y,s))\Gamma_\nu(x,y)\ dy,\]
for $\nu\in\{\mu,\varepsilon\}$. The exponential function is Lipschitz continuous on compact sets and due to the boundedness of the fractions occurring in the nonlocal productivity-production operator, we can estimate

\begin{equation*}
\begin{split}
&\left\| \exp\left(\frac{\Phi_\mu(v_1)(\cdot,s)}{\Phi_\varepsilon(v_1)(\cdot,s)+\xi}\ s\right) - \exp\left(\frac{\Phi_\mu(v_2)(\cdot,s)}{\Phi_\varepsilon(v_2)(\cdot,s)+\xi}\ s\right)\right\|_{L^2( \Omega)}\\[4mm]
\le\ & L_{exp}\ s \left\| \frac{\Phi_\mu(v_1)(\cdot,s)}{\Phi_\varepsilon(v_1)(\cdot,s)+\xi} - \frac{\Phi_\mu(v_2)(\cdot,s)}{\Phi_\varepsilon(v_2)(\cdot,s)+\xi}\  \right\|_{L^2( \Omega)}\\[4mm]
=\ & L_{exp}\ s \left\| \frac{ \Phi_\mu(v_1)(\cdot,s) \left( \Phi_\varepsilon(v_2)(\cdot,s) + \xi \right)}{\left(\Phi_\varepsilon(v_1)(\cdot,s)+\xi\right)\left( \Phi_\varepsilon(v_2)(\cdot,s) + \xi \right)} 
- \frac{\left(\Phi_\varepsilon(v_1)(\cdot,s)+\xi\right)\Phi_\mu(v_2)(\cdot,s)}{\left(\Phi_\varepsilon(v_1)(\cdot,s)+\xi\right)\left( \Phi_\varepsilon(v_2)(\cdot,s) + \xi \right)}\  \right\|_{L^2( \Omega)}\\[4mm]
\le\ & L_{exp}\ s \left(\int_\Omega\left[\left| \frac{ \Phi_\mu(v_1)(x,s) \Phi_\varepsilon(v_2)(x,s) }{\left(\Phi_\varepsilon(v_1)(x,s)+\xi\right)\left( \Phi_\varepsilon(v_2)(x,s) + \xi \right)} \right.\right. \right. \\
& \left.\left. \left. \hspace*{7.4cm} - \frac{\Phi_\mu(v_1)(x,s)\Phi_\varepsilon(v_1)(x,s)}{\left(\Phi_\varepsilon(v_1)(x,s)+\xi\right)\left( \Phi_\varepsilon(v_2)(x,s) + \xi \right)}\right| \right.\right.\\[3mm]
& \hspace*{.3cm} + \left. \left. \left| \frac{ \Phi_\mu(v_1)(x,s) \Phi_\varepsilon(v_1)(x,s) }{\left(\Phi_\varepsilon(v_1)(x,s)+\xi\right)\left( \Phi_\varepsilon(v_2)(x,s) + \xi \right)} 
- \frac{\Phi_\mu(v_2)(x,s)\Phi_\varepsilon(v_1)(x,s)}{\left(\Phi_\varepsilon(v_1)(x,s)+\xi\right)\left( \Phi_\varepsilon(v_2)(x,s) + \xi \right)}\right| \right.\right.\\[3mm]
& \hspace*{.3cm} + \left. \left.\left| \frac{\xi\left( \Phi_\mu(v_1)(x,s) - \Phi_\mu(v_2)(x,s)\right)}{\left(\Phi_\varepsilon(v_1)(x,s)+\xi\right)\left( \Phi_\varepsilon(v_2)(x,s) + \xi \right)}\right| \right]^2\ dx\right)^{\frac{1}{2}}\\[4mm]
\le \ & L_{exp}\ s \bigg( \int_{\Omega} \left( \frac{1}{\xi} \left| \Phi_\varepsilon(v_1)(x,s) - \Phi_\varepsilon(v_2)(x,s) \right| + \frac{2}{\xi} \left| \Phi_\mu(v_1)(x,s)- \Phi_\mu(v_2)(x,s) \right|  \right)^2\  dx \bigg)^{\frac{1}{2}},
\end{split}
\end{equation*}

since

\begin{equation*}
\begin{split}
\left|\frac{ \Phi_\mu(v_1)(x,s)  }{\left(\Phi_\varepsilon(v_1)(x,s)+\xi\right)\left( \Phi_\varepsilon(v_2)(x,s) + \xi \right)}\right| \le 
 \left|\frac{\Phi_\varepsilon(v_1)(x,s)}{\left(\Phi_\varepsilon(v_1)(x,s)+\xi\right)\left( \Phi_\varepsilon(v_2)(x,s) + \xi \right)}\right|\le \frac{1}{\xi} ,
\end{split}
\end{equation*}

and 
\begin{equation*}
\begin{split}
 \left|\frac{ \xi  }{\left(\Phi_\varepsilon(v_1)(x,s)+\xi\right)\left( \Phi_\varepsilon(v_2)(x,s) + \xi \right)}\right| \le \frac{1}{\xi},
\end{split}
\end{equation*}

for all $v_1,v_2\in \mathcal{C}(0,T^*;V_c(\Omega\cup\Omega_{\mathcal{I}}) )$ and $(x,s)\in \Omega\times[0,T]$. \\

Applying the Minkowski and the H\"older inequalities and exploiting the Lipschitz continuity of $\phi$, we end up with

\begin{equation*}
\begin{split}
&\left\| \exp\left(\frac{\int_{\Omega\cup\Omega_{\mathcal{I}}}\phi(v_1(y,s))\Gamma_{\mu}(\cdot,y)\  dy}{\int_{\Omega\cup\Omega_{\mathcal{I}}}\phi(v_1(y,s))\Gamma_{\varepsilon}(\cdot, y)\   dy+\xi}\ s\right) \right.\\
& \left. \hspace*{6cm}- \exp\left(\frac{\int_{\Omega\cup\Omega_{\mathcal{I}}}\phi(v_2(y,s)) \Gamma_{\mu}(\cdot,y)\  dy}{\int_{\Omega\cup\Omega_{\mathcal{I}}} \phi(v_2(y,s)) \Gamma_{\varepsilon}(\cdot,y) \ dy+\xi}\ s\right) \right\|_{L^2( \Omega)}\\[3mm]
&\le s K\ \|v_1(\cdot,s)-v_2(\cdot,s)\|_{L^2(\Omega\cup\Omega_{\mathcal{I}})}
\end{split}
\end{equation*}
with 

\[K:=\frac{1}{\xi} \left(L_{exp}L_\phi\|\Gamma_{\varepsilon}\|_{L^2(\Omega\times (\Omega\cup\Omega_{\mathcal{I}}))} + 2L_{exp}L\phi\|\Gamma_{\mu}\|_{L^2(\Omega\times (\Omega\cup\Omega_{\mathcal{I}}))}\right)<\infty .\]

Note that on the bounded domain $\Omega\cup\Omega_{\mathcal{I}}$, we have
 \[\|\Gamma_{\nu}\|_{L^2(\Omega\times(\Omega\cup\Omega_{\mathcal{I}}))}<\infty\]
for all $\nu>0$. Thus, we can estimate
\begin{equation*}
\begin{split}
&\|P(v_1)(\cdot,s)-P(v_2)(\cdot,s)\|_{L^2(\Omega)}\le s K \|A_0\|_{L^\infty(\Omega)} \|v_1(\cdot,s)-v_2(\cdot,s)\|_{L^2(\Omega\cup \Omega_{\mathcal{I}})}.
\end{split}
\end{equation*}

Combining both estimates for the left- and right-side of the PIDE and applying Young's inequality for a constant $\beta>0$, we get
\begin{equation*}
\begin{split}
& \frac{1}{2} \|k_1-k_2(t)\|^2_{L^2(\Omega)} + \int_0^t |||k_1-k_2(s)|||^2 ds\\
& \hspace*{4cm}\le  \int_0^t \frac{L(s)^2}{2\beta} \|v_1-v_2(s)\|^2_{L^2(\Omega\cup\Omega_{\mathcal{I}})} + \frac{\beta}{2} \|k_1-k_2\|^2_{L^2(\Omega)} ds 
\end{split}
\end{equation*}
with $L(s):=\|A_0\|_{L^\infty(\Omega)}(L_p\exp(s)+MKs).$\\

We choose $2C_1 < \beta < 2C_1\left(\frac{1}{2C_2}+1\right),$ where $C_1$ and $C_2$ are the constants from (\ref{norm_equivalnece}). Then, again with (\ref{norm_equivalnece}), we can interpret the inequality in terms of the $V_c(\Omega\cup\Omega_{\mathcal{I}})$-norm as follows
\begin{equation*}
\begin{split}
& \frac{1}{2C_2} |||k_1-k_2(t)|||^2 \le  \int_0^t\left( \frac{L(s)^2}{2\beta C_1} |||v_1-v_2(s)|||^2 + \left(\frac{\beta}{2C_1}-1\right) |||k_1-k_2|||^2\right)\  ds .
\end{split}
\end{equation*}
Note that we have once more exploited the Dirichlet volume constraints in order to rewrite
\[\|k_1-k_2(t)\|_{L^2(\Omega)} = \|k_1-k_2(t)\|_{L^2(\Omega\cup\Omega_{\mathcal{I}})}\]
Taking the maximum over all $t\in[0,T^*]$ and sorting the terms, we have
\begin{equation*}
\begin{split}
&\frac{1}{2C_2}\|k_1-k_2\|^2_{L^\infty(0,T^*;V_c(\Omega\cup\Omega_{\mathcal{I}}))} + (1-\frac{\beta}{2C_1})T^* \|k_1-k_2\|^2_{L^\infty(0,T^*;V_c(\Omega\cup\Omega_{\mathcal{I}}))}\\[3mm]
& \hspace*{6cm} \le C(T^*) \|v_1-v_2\|^2_{L^\infty(0,T^*; V_c(\Omega\cup\Omega_{\mathcal{I}}))},
\end{split}
\end{equation*} 
where 
\[C(T^*):=\frac{1}{2\beta C_1} ess \sup_{t\in[0,T^*]}\int_0^t L(s)^2 ds.\]
Taking the limit $T^*\to 0$, we obtain $C(T^*)\to 0$ since

\begin{equation*}
\begin{split}
\int_0^t L(s)^2ds &\le \int_0^{T^*} L(s)^2ds  \le \tilde{C} \int_0^{T^*}(\exp(s)+ s)^2 ds \\
&= \tilde{C}\left(\frac{1}{2}\exp(2T^*)- \frac{1}{2}+\frac{1}{3}T^{*^2} + \exp(T^*)(T^*-1)+1\right)\to 0.
\end{split}
\end{equation*}
Thus, we conclude that there exists a $T^*$ small enough such that
\[\frac{C(T^*)}{\left(\frac{1}{2 C_2}+(1-\frac{\beta}{2 C_1})T^*\right)}<1.\]
Note that in particular
\[\left(\frac{1}{2 C_2}+(1-\frac{\beta}{2 C_1})T^*\right)>0\]
for $T^*\le 1$ by the choice of $\beta$. Hence, we have shown that $\mathcal{S}$ is a contraction on a sufficiently small time interval. According to Banach's fixed point theorem, $\mathcal{S}$ has a unique fixed point on every bounded set. Since the local solution $k$ is independent of the time horizon $T^*$, we can proceed on the interval $[T^*,2T^*]$ using the same arguments as above but with a new initial condition $k(\cdot,T^*)$. After finitely many steps, we can construct a weak solution of (\ref{PIDE_weak}) on the whole time space cylinder after finitely many steps.
Moreover, this solution is unique.
\end{proof}

Now, we have a closer look at the regularity of the weak solution $k$. We start calculating an a priori estimate, which depends only on the initial value condition and the inhomogeneity.

\begin{corollary}\label{a_priori_bounded}
There exists a constant $C_\infty>0$ independent of the data $c$ and $k_0$ such that the weak solution of (\ref{cons_nonlocal}) satisfies the following a priori  estimate
\[\|k\|_{H^1(0,T;V_c(\Omega\cup\Omega_{\mathcal{I}}))} \le C_{\infty}\ \left(\|c\|_{L^2(0,T;L^2(\Omega))} + \|k_0\|_{L^2(\Omega)} +1  \right).\]
\end{corollary}

In particular, this estimate gives us the continuity of the solution operator 
\[G:L^2(0,T;L^2(\Omega)) \times L^2(\Omega) \to H^1(0,T;V_c(\Omega\cup\Omega_{\mathcal{I}})) \]
that maps any inhomogeneity $c$ and initial condition $k_0$ to the solution of (\ref{cons_nonlocal}) (cf. \citealp[p.112]{troltzsch}).

\begin{proof}
First, we recall that 
\[\|k\|^2_{H^1(0,T;V_c(\Omega\cup\Omega_{\mathcal{I}}))} = \|k\|^2_{L^2(0,T;V_c(\Omega\cup\Omega_{\mathcal{I}}))}+\|k_t\|^2_{L^2(0,T;V'_c(\Omega\cup\Omega_{\mathcal{I}}))}. \]
We estimate the first term exploiting the coercivity of the bilinear form ${\bf a}$. We choose a $t\in[0,T]$ and derive the weak formulation of the capital equation for the test function $k\in H^1(0,T;V_c(\Omega\cup\Omega_{\mathcal{I}}))\cap \mathcal{C}(0,T;V_c(\Omega\cup\Omega_{\mathcal{I}}))$ which yields
\begin{equation*}
\begin{split}
\int_0^t \int_\Omega \frac{\partial k}{\partial t} k\ dxds + \int_0^t {\bf a}(k,k) ds = \int_0^t\int_\Omega (\mathcal{P}(k)-c)k\ dxds.
\end{split}
\end{equation*}
As already proven, we have ${\bf a}(k,k)\ge |||k|||^2$. Hence, we can estimate the left-hand side as

\begin{equation*}
\begin{split}
LHS &= \int_0^t \int_\Omega \frac{\partial k}{\partial t} k\ dxds + \int_0^t {\bf a}(k,k) ds\\
&\ge \int_0^t\int_\Omega\frac{\partial k}{\partial t} \ dxds + \int_0^t |||k(s)|||^2 ds\\
&= \frac{1}{2}\|k(t)\|^2_{L^2(\Omega)} - \frac{1}{2}\|k_0\|^2_{L^2(\Omega)} + \int_0^t |||k(s)|||^2 ds\\
& = \frac{1}{2}\|k(t)\|^2_{L^2(\Omega)} - \frac{1}{2}\|k_0\|^2_{L^2(\Omega)} + \|k\|_{L^2(0,t;V_c(\Omega\cup\Omega_{\mathcal{I}}))}^2 
\end{split}
\end{equation*} 
for all $t\in[0,T]$. In order to derive an upper bound for the right-hand side, we exploit the Lipschitz continuity and the boundedness of the production function $p$, $p(0)=0$, and the boundedness of the fraction
\[\frac{\int_{\Omega\cup\Omega_{\mathcal{I}}} \phi(k(y,t))\Gamma_\mu(x,y)\ dy}{\int_{\Omega\cup\Omega_{\mathcal{I}}} \phi(k(y,t))\Gamma_\varepsilon(x,y)\ dy + \xi}\le 1 \]
for all $x\in\Omega$ and $t\in[0,T]$. With these properties, we get

\begin{equation*}
\begin{split}
RHS &= \int_0^t\int_\Omega (\mathcal{P}(k)-c)k\ dxds \\
&\le \int_0	^t \|\mathcal{P}(k)(s)\|_{L^2(\Omega)}\|k(s)\|_{L^2(\Omega)} ds + \int_0^t \|c(s)\|_{L^2(\Omega)}\|k(s)\|_{L^2(\Omega)} ds\\
& \le \|A_0\|_{L^\infty(\Omega)} \int_0^t \left( \int_\Omega\left| e^sp(k(s))\right|^2\ dx\right)^{\frac{1}{2}} \left( \int_\Omega|k(s)|^2\ dx\right)^{\frac{1}{2}} ds \\
&+ \int_0^t \|c(s)\|_{L^2(\Omega)}\|k(s)\|_{L^2(\Omega)} ds\\
&\le  \|A_0\|_{L^\infty(\Omega)} M_p |\Omega| \left( \int_0^t e^{2s}ds\right)^{\frac{1}{2}} \left( \int_0^t \|k(s)\|^2_{L^2(\Omega)}ds\right)^{\frac{1}{2}} \\
&+ \int_0^t \|c(s)\|_{L^2(\Omega)}\|k(s)\|_{L^2(\Omega)} ds\\
& \le \|A_0\|_{L^\infty(\Omega)} M_p |\Omega| \left( \frac{e^{2t}}{2} - \frac{1}{2} \right)^{\frac{1}{2}}\|k\|_{L^2(0,t;L^2(\Omega))}+ \int_0^t \|c(s)\|_{L^2(\Omega)}\|k(s)\|_{L^2(\Omega)} ds.\\
\end{split}
\end{equation*}
Using Young's inequality for two constants $\eta_1,\eta_2>0$, we have
\begin{equation*}
\begin{split}
RHS&\le \|A_0\|_{L^\infty(\Omega)} M_p|\Omega| \left( \frac{\eta_1}{2} \left( \frac{e^{2t}}{2} - \frac{1}{2}\right) + \frac{1}{2\eta_1}\|k\|_{L^2(0,t;L^2(\Omega))}^2\right) \\
& \quad + \frac{\eta_2}{2}\|c\|^2_{L^2(0,t;L^2(\Omega))} + \frac{2}{\eta_2} \|k\|_{L^2(0,t;L^2(\Omega))}^2\\
&\le \|A_0\|_{L^\infty(\Omega)} M_p|\Omega| \left( \frac{\eta_1}{2} \left( \frac{e^{2t}}{2} - \frac{1}{2}\right) + \frac{1}{2C_1\eta_1}\|k\|_{L^2(0,t;V_c(\Omega\cup\Omega_{\mathcal{I}}))}^2\right) \\
& \quad + \frac{\eta_2}{2}\|c\|^2_{L^2(0,t;L^2(\Omega))} + \frac{1}{2C_1\eta_2}\|k\|_{L^2(0,t;V_c(\Omega\cup\Omega_{\mathcal{I}}))}^2
\end{split}
\end{equation*}
by the equivalence of spaces. Combining both estimates yields
\begin{equation*}
\begin{split}
& \frac{1}{2}\|k(t)\|^2_{L^2(\Omega)} + \|k\|_{L^2(0,t;V_c(\Omega\cup\Omega_{\mathcal{I}}))}^2 \\
&\le  \frac{1}{2}\|k_0\|^2_{L^2(\Omega)} + \|A_0\|_{L^\infty(\Omega)} M_p|\Omega| \left( \frac{\eta_1}{2} \left( \frac{e^{2t}}{2} - \frac{1}{2}\right) + \frac{1}{2C_1\eta_1} \|k\|_{L^2(0,t;V_c(\Omega\cup\Omega_{\mathcal{I}}))}^2\right) \\
& \quad + \frac{\eta_2}{2}\|c\|^2_{L^2(0,t;L^2(\Omega))} + \frac{1}{2C_1\eta_2} \|k\|_{L^2(0,t;V_c(\Omega\cup\Omega_{\mathcal{I}}))}^2.
\end{split}
\end{equation*}
Taking the maximum over all $t\in[0,T]$ finally gives us
\begin{equation*}
\frac{1}{2} \|k\|^2_{L^\infty(0,T;L^2(\Omega))} + \hat{c} \|k\|^2_{L^2(0,T;V_c(\Omega\cup\Omega_{\mathcal{I}}))} \le C\left( \|k_0\|_{L^2(\Omega)} +  \|c\|_{L^2(0,T;L^2(\Omega))} + 1\right)^2
\end{equation*}
with 
\[\hat{c}:=1 - \frac{ \|A_0\|_{L^\infty(\Omega)} M_p|\Omega|}{2C_1\eta_1} - \frac{1}{2C_1\eta_2}>0\]
for $\eta_1,\eta_2>0$ sufficiently large.\\
In order to estimate the second term, we define some linear functionals analogously to \citet[pp.119]{troltzsch}, namely

\begin{equation*}
\begin{split}
&F_1(t):v\mapsto \langle k(t),v\rangle_{V(\Omega\cup\Omega_{\mathcal{I}})} \hspace*{2cm} F_2(t): v\mapsto \langle \delta k(t),v\rangle_{L^2(\Omega)}\\
&F_3(t):v\mapsto \langle \mathcal{P}(k)(t),v\rangle_{L^2(\Omega)}\hspace*{2cm} F_4(t):v\mapsto \langle c(t),v\rangle_{L^2(\Omega)}
\end{split}
\end{equation*}
These functionals are continuous since
\begin{equation*}
\begin{split}
&|F_1(t)v| \le |||k(t)|||\ |||v|||\quad \mbox{ and }\quad |F_2(t)v| \le \delta |||k(t)|||\ |||v|||
\end{split}
\end{equation*}
using the Cauchy-Schwartz inequality. 
For the third and fourth functional we get
\[|F_3(t)v|\le \mathrm{const}(t) |||k(t)|||\ |||v|||\quad \mbox{ and }\quad |F_4(t)v|\le \|c(t)\|_{L^2(\Omega)}|||v|||\]
using the estimates of the proof of Lemma \ref{a_properties}. Here we denote by $\mathrm{const}(t)$ a constant depending only on $t$. For fixed $k$ and $c$, we can interpret the values $ |||k(t)|||$ and $\|c(t)\|_{L^2(\Omega)}$ as constants of the definition of the continuity of $F_i,\ i=1,...,4$. According to \citet[p.120]{troltzsch}, we can find a constant $\hat{c}$ such that
\[\|F_i(t)\|_{V_c(\Omega\cup\Omega_{\mathcal{I}})'} \le \hat{c} |||k(t)|||,\ i=1,2,3 \quad \mbox{ and } \|F_4(t)\|_{V_c'(\Omega\cup\Omega_{\mathcal{I}})} \le \hat{c}\|c(t)\|_{L^2(\Omega)}.\]

From the weak formulation, we know that
\[\|k_t\|^2_{L^2(0,T;V_c'(\Omega\cup\Omega_{\mathcal{I}}))} \le \sum_{i=1}^4 \|F_i\|_{L^2(0,T;V_c'(\Omega\cup\Omega_{\mathcal{I}}))}.\]
Using the estimation for $k$,
\[\|k\|^2_{L^2(0,T;V_c(\Omega\cup\Omega_{\mathcal{I}}))} \le C\left( \|k_0\|_{L^2(\Omega)} +  \|c\|_{L^2(0,T;L^2(\Omega))} + 1\right)^2,\] 
we have 
\[\|k_t\|^2_{L^2(0,T;V_c'(\Omega\cup\Omega_{\mathcal{I}}))} \le \tilde{C}\left( \|k_0\|_{L^2(\Omega)} +  \|c\|_{L^2(0,T;L^2(\Omega))} + 1\right)^2.\]
Summing up both estimates, we finally achieve
\[ \|k\|^2_{L^2(0,T;V_c(\Omega\cup\Omega_{\mathcal{I}}))} + \|k_t\|^2_{L^2(0,T;V_c'(\Omega\cup\Omega_{\mathcal{I}}))} \le C^2_\infty \left( \|k_0\|_{L^2(\Omega)} +  \|c\|_{L^2(0,T;L^2(\Omega))} + 1\right)^2 \]
which completes the proof.
\end{proof}

So far, we have only considered the initial data and the right-hand side of the PIDE to be $L^2(\Omega\times[0,T])$ functions. The highest regularity, we can achieve in that case, is $\mathcal{C}(0,T;L^2(\Omega\cup \Omega_{\mathcal{I}}))$. We cannot expect a higher regularity in the space direction, since there is no operator, such as the differential operator, that drives regularity. Nevertheless, we would expect a higher regularity of the weak solution, whenever we choose a higher regularity for the data. The following theorem shows, that the weak solution of the nonlocal capital accumulation has indeed the same regularity as the data.

\begin{theorem}\label{LinftyOmega}
Let all assumptions of Theorem \ref{Existence_semilinear_bounded} hold, and let $k_0\in L^\infty(\Omega)$ and $c\in L^\infty(\Omega\times[0,T])$. Then, the weak solution of the capital accumulation equation (\ref{cons_nonlocal}) is $\mathcal{C}(0,T;V_c(\Omega\cup\Omega_{\mathcal{I}}))\cap H^1(0,T;L^\infty(\Omega))$.
\end{theorem}

\begin{remark}\label{Vinfty}
By the intersection $\mathcal{C}(0,T;V_c(\Omega\cup\Omega_{\mathcal{I}}))\cap H^1(0,T;L^\infty(\Omega))$, we mean a subspace of the $\mathcal{C}([0,T];V_c(\Omega\cup\Omega_{\mathcal{I}}))$ space. We define the Banach space
\begin{equation*}
\mathcal{V}^\infty:=\{u\in \mathcal{C}([0,T];V_c(\Omega\cup\Omega_{\mathcal{I}})):\ ess\sup_{(x,t)\in\Omega\times(0,T)}|u(x,t)|<\infty\}
\end{equation*}
endowed with the norm
\[\|u\|_{\mathcal{V}^\infty}:=\|u\|_{\mathcal{C}([0,T];V_c(\Omega\cup\Omega_{\mathcal{I}}))} + \|u\|_{L^\infty(\Omega\times(0,T))}\]
and refer to $\mathcal{V}^\infty$ whenever we consider the intersection space. 
\end{remark}

\begin{proof}
Consider the solution $k^*\in \mathcal{C}([0,T];V_c(\Omega\cup\Omega_{\mathcal{I}}))$ of the capital accumulation equation. For such $k^*$, the production-productivity operator $\mathcal{P}$ maps to $L^\infty$ since we have assumed $A_0$ to be a $L^\infty(\Omega)$ function and the production function $p$ to be bounded. We can calculate 
\[\|\mathcal{P}(k^*)\|_{L^\infty(\Omega\times[0,T])} = \|A_0\|_{L^\infty(\Omega)}e^T \ M_p,\]
where $M_p$ denotes the uniform upper bound of $p$. Moreover, we know that 
\[\int_{\Omega\cup\Omega_{\mathcal{I}}} \Gamma_\varepsilon(x,y)\ dy =: \hat{\Gamma}_\varepsilon(x)\le 1\]
and using H\"older's inequality, it follows that
\[\int_{\Omega\cup\Omega_{\mathcal{I}}} k^*(y,t)\Gamma_\varepsilon(x,y)\ dy < \infty \]
for all $x\in\Omega$. For a fixed $x\in\Omega$, we consider the capital accumulation equation
\begin{equation*}
\begin{split}
\frac{\partial k}{\partial t}(x,t)-\mathcal{NL}(k)(x,t) + \delta k(x,t) -\mathcal{P}(k)(x,t) & = -c(x,t)\hspace*{0.8cm} \mbox{ on }(0,T)\\
k(x,0) &= k_0(x).
\end{split}
\end{equation*}
We rewrite the equation as
\[\frac{\partial k}{\partial t}(x,t) + (\hat{\Gamma}_\varepsilon(x) + \delta) k(x,t) = \int_{\Omega\cup \Omega_{\mathcal{I}}}k(y,t)\Gamma_\varepsilon(x,y)\ dy + \mathcal{P}(k)(x,t)-c(x,t) \mbox{ on }(0,T).\]
We neglect the dependence of $k$ of the right-hand side, since it maps every $k$ to $L^\infty$, and define
\[g_x(t):= \int_{\Omega\cup \Omega_{\mathcal{I}}}k(y,t)\Gamma_\varepsilon(x,y)\ dy + \mathcal{P}(k)(x,t)-c(x,t) \in L^\infty(\Omega\times[0,T]).\] 
Note, that the regularity of $g_x$ is determined by the regularity of $c$. 
Now, we consider the inhomogeneous linear ordinary differential equation
\begin{equation*}
\frac{\partial k_x}{\partial t} + (\hat{\Gamma}_{x,\varepsilon} +\delta)\ k = g_x
\end{equation*} 
depending on the parameter $x$. We know that $g_x$ is continuous in $t$. Hence, the equation has a solution $\overline{k}$ given as
\[\overline{k}_x(t) = e^{-t(\hat{\Gamma}_{x,\varepsilon} +\delta)}\left( k_{0,x} + \int_0^t g_x(s)e^{s(\hat{\Gamma}_{x,\varepsilon} +\delta)}ds\right)\]
which is bounded for every $x\in\Omega$. Thus, we conclude
\[\overline{k}\in L^\infty(\Omega\times[0,T]).\]
Since the solution of the capital accumulation equation is unique we get  $k^*=\overline{k}$, which ends the proof.
\end{proof}

\bibliography{literature}
\end{document}